\documentclass{amsart}
\usepackage{amsfonts, amsbsy, amsmath, amssymb}
\usepackage{mathdots}

\newtheorem{thm}{Theorem}[section]
\newtheorem{lem}[thm]{Lemma}
\newtheorem{cor}[thm]{Corollary}

\newtheorem{rmk}[thm]{Remark}
\newtheorem{fac}[thm]{Fact}

\newtheorem{thm-con}[thm]{Theorem-Conjecture}
\numberwithin{equation}{section}

\theoremstyle{definition}
\newtheorem{defn}[thm]{Definition}

\allowdisplaybreaks

\newcommand\numberthis{\addtocounter{equation}{1}\tag{\theequation}}

\newcommand{\f}{\Bbb F}
\newcommand{\uu}{\boldsymbol u}
\newcommand{\vv}{\boldsymbol v}
\newcommand{\x}{\boldsymbol x}
\newcommand{\X}{\boldsymbol X}
\newcommand{\Y}{\boldsymbol Y}
\newcommand{\T}{\boldsymbol t}

\begin{document}

\title{Polynomials Meeting Ax's Bound}

\author[Xiang-dong Hou]{Xiang-dong Hou}
\address{Department of Mathematics and Statistics,
University of South Florida, Tampa, FL 33620}
\email{xhou@usf.edu}



\begin{abstract}
Let $f\in\Bbb F_q[X_1,\dots,X_n]$ with $\deg f=d>0$ and let $Z(f)=\{(x_1,\dots,x_n)\in \Bbb F_q^n: f(x_1,\dots,x_n)=0\}$. Ax's theorem states that $|Z(f)|\equiv 0\pmod {q^{\lceil n/d\rceil-1}}$, that is, $\nu_p(|Z(f)|)\ge m(\lceil n/d\rceil-1)$, where $p=\text{char}\,\Bbb F_q$, $q=p^m$, and $\nu_p$ is the $p$-adic valuation. In this paper, we determine a condition on the coefficients of $f$ that is necessary and sufficient for $f$ to meet Ax's bound, that is,
$\nu_p(|Z(f)|)=m(\lceil n/d\rceil-1)$. Let $R_q(d,n)$ denote the $q$-ary Reed-Muller code $\{f\in\Bbb F_q[X_1,\dots,X_n]: \deg f\le d,\ \deg_{X_j}f\le q-1,\ 1\le j\le n\}$, and let $N_q(d,n;t)$ be the number of codewords of $R_q(d,n)$ with weight divisible by $p^t$.
As applications of the aforementioned result, we find explicit formulas for $N_q(d,n;t)$ in the following cases: (i) $q=2^m$, $n$ even, $d=n/2$, $t=m+1$; (ii) $q=2$, $n/2\le d\le n-2$, $t=2$; (iii) $q=3^m$, $d=n$, $t=1$; (iv) $q=3$, $n\le d\le 2n$, $t=1$.
\end{abstract}
\maketitle

\section{Introduction}

Let $\f_q$ be the finite field with $q=p^m$ elements, where $p=\text{char}\,\f_q$. Let $f\in\f_q[X_1,\dots,X_n]$ with $\deg f=d>0$ and let $Z(f)=\{(x_1,\dots,x_n)\in\f_q^n:f(x_1,\dots,x_n)=0\}$. Ax's theorem \cite{Ax-AJM-1964} states that
\begin{equation}\label{1.1}
\nu_p(|Z(f)|)\ge m\Bigl(\Bigl\lceil\frac nd\Bigr\rceil-1\Bigr),
\end{equation}
where $\nu_p$ denotes the $p$-adic valuation. Ax's theorem is a strengthening of a result by  Warning \cite{Warning-AMSUH-1936}. Further back along this line were a conjecture by  Artin on the existence of nonzero roots of a homogeneous polynomial $f\in\f_q[X_1,\dots,X_n]$ with $n>\deg f$ and  Chevalley's proof of Artin's conjecture; see \cite{Chevalley-AMSUH-1936}.

The main ingredient of the original proof of Ax's theorem is the Stickelberger congruence of Gauss sums. A different proof based on the same idea but without using Gauss sums and the Stickelberger congruence was given by Ward \cite{Ward-DM-1990}.

Ax's theorem has been extended to several polynomials by N. Katz \cite{Katz-AJM-1971}. Assume that $f_i\in\f_q[X_1,\dots,X_n]$, $1\le i\le r$, are such that $\deg f_i=d_i>0$ and $d_1=\max_{1\le i\le r}d_i$, then
\begin{equation}\label{1.2}
\nu_p\bigl(|Z(f_1)\cap\cdots\cap Z(f_r)|\bigr)\ge m\Bigl\lceil\frac{n-d_1-\cdots-d_r}{d_1}\Bigr\rceil.
\end{equation}
The original proof of Katz's theorem relied on sophisticate tools. A simpler proof was given by Wan \cite{Wan-AJM-1989, Wan-PAMS-1995} using a method similar to Ax's. A more elementary proof of Katz's theorem for prime fields was found by Wilson \cite{Wilson-DM-2006}. Sun \cite{Sun-arxiv-0608560} further extended Katz's theorem for prime fields along the line of Wilson's approach.

Delsarte and McEliece \cite{Delsarte-McEliece-AJM-1976} studied functions from a finite abelian group $A$ to $\f_q$, where $\text{gcd}(|A|,q)=1$. Such functions were treated as elements of the group algebra $\f_q[A]$. Instead of polynomials in $\f_q[X_1,\dots,X_n]$ with a given degree, functions $f:A\to\f_q$ that belong to an ideal of $\f_q[A]$ were considered. (In coding theory, an ideal of $\f_q[A]$ is called an {\em abelian code}.) \cite{Delsarte-McEliece-AJM-1976} established a lower bound for $\nu_p(|Z(f)|)$, which implies Ax's theorem when $A$ is the cyclic group of order $q^n-1$. D. Katz \cite{Katz-DCC-2012} generalized the result of \cite{Delsarte-McEliece-AJM-1976} to a lower bound for $\nu_p(|Z(f_1)\cap\cdots\cap Z(f_r)|)$, $f_1,\dots,f_r\in\f_q[A]$, and when $A$ is the cyclic group of order $q^n-1$, the generalized bound gives the theorem of N. Katz.

Although not obvious, \eqref{1.2} {\em actually follows from} \eqref{1.1}, which was a finding by the author \cite{Hou-FFA-2005}.

The bounds in \eqref{1.1} and \eqref{1.2} are both sharp; see \cite{Ax-AJM-1964, Katz-AJM-1971}. Therefore, improvements of these bounds are possible only under additional assumptions. For such improvements, see Cao \cite{Cao-JNT-2012}, Cao and Sun \cite{Cao-Sun-JNT-2007}, and O. Moreno and C. Moreno \cite{Moreno-Moreno-AJM-1995}.

Focusing on \eqref{1.1}, we note that another way to ``improve'' the bound is to find the next term in the $p$-adic expansion of $|Z(f)|$. In this paper, we will find an expression $E(f)\in\f_p$ such that
\begin{equation}\label{1.3}
|Z(f)|\equiv q^{\lceil n/d\rceil-1}E(f)\pmod{p^{m(\lceil n/d\rceil-1)+1}}.
\end{equation}
Therefore,
\[
\nu_p(|Z(f)|)\ge m\Bigl(\Bigl\lceil\frac nd\Bigr\rceil-1\Bigr)+1
\]
if and only if $E(f)=0$. The expression $E(f)$ is a homogeneous polynomial over $\f_p$ in the coefficients of $f$; it is not explicit in general. However, in several special but nontrivial cases, $E(f)$ can be made explicit. By exploiting this fact, we obtain several explicit formulas for the number of codewords in a Reed-Muller code with weight divisible by a power of $p$. More precisely, let $R_q(d,n)$ denote the $q$-ary Reed-Muller code $\{f\in\f_q[X_1,\dots,X_n]: \deg f\le d,\ \deg_{X_j}f\le q-1,\ 1\le j\le n\}$, where $\deg$ is the total degree and $\deg_{X_j}$ is the degree in $X_j$, and let $N_q(d,n;t)$ be the number of codewords of $R_q(d,n)$ with weight divisible by $p^t$, where $p=\text{char}\,\f_q$.
We find explicit formulas for $N_q(d,n;t)$ in the following cases: (i) $q=2^m$, $n$ even, $d=n/2$, $t=m+1$; (ii) $q=2$, $n/2\le d\le n-2$, $t=2$; (iii) $q=3^m$, $d=n$, $t=1$; (iv) $q=3$, $n\le d\le 2n$, $t=1$.

In fact, for a finite abelian group $A$ and $f\in\f_q[A]$, Delsarte and McEliece had found a formula for the next term in the $p$-adic expansion of $|Z(f)|$; see \cite[(4.29)]{Delsarte-McEliece-AJM-1976}. From that formula with $A=\Bbb Z/(q^n-1)\Bbb Z$, one can derive a expression for the ``next term'' in Ax's theorem. The formula for the ``next term'' in \cite{Delsarte-McEliece-AJM-1976}, including the case $A=\Bbb Z/(q^n-1)\Bbb Z$, involves the Fourier transform of $f$ which takes values in an extension of $\f_q$. In comparison, the expression $E(f)$ determined in \eqref{2.14.5} of the present paper is considerably simpler. 

In Section~2, we determine the expression $E(f)$ in \eqref{1.3}. The method is a refinement of the original proof of Ax's theorem and relies on a careful analysis of the Stickelberger congruence of Gauss sums. Applications to Reed-Muller codes are discussed in Section~3.

Throughout the paper, for $\boldsymbol u, \boldsymbol v\in\Bbb Z^n$, the relations $\boldsymbol u\equiv \boldsymbol v\pmod k$ and $\boldsymbol u\le \boldsymbol v$ are meant to be component wise. We define
\begin{equation}\label{Delta}
\Delta_n=\left[
\begin{matrix}
0&&1\cr
&\iddots\cr
1&&0
\end{matrix}\right]_{n\times n}.
\end{equation}


\section{$p$-adic Expansion of $|Z(f)|$}

\subsection{Gauss sum and Stickelberger congruence}\

Facts gathered in this subsection can be found in any textbook on algebraic number theory, e.g., Lang~\cite[Ch.\,IV, \S3]{Lang}.

For an integer $k>0$, let $\zeta_k=e^{2\pi i/k}$. The ring of integers of a number field $F$ is denoted by $\frak o_F$. Let $p$ be a rational prime, $m>0$ and $q=p^m$. Let $\frak p$ be a prime of $\frak o_{\Bbb Q(\zeta_{q-1})}$ lying above $p$. $\frak p$ is unramified over $p$ and $\frak o_{\Bbb Q(\zeta_{q-1})}/\frak p=\f_q$. The Teichm\"uller set $T=\{0\}\cup\langle\zeta_{q-1}\rangle=\{0,\zeta_{q-1}^0,\dots,\zeta_{q-1}^{q-2}\}$ forms a system of coset representative of $\frak p$ in $\frak o_{\Bbb Q(\zeta_{q-1})}$, that is, $\f_q=\frak o_{\Bbb Q(\zeta_{q-1})}/\frak p=\{t+\frak p:t\in T\}$. The Teichm\"uller character $\chi_\frak p$ is a multiplicative character of $\f_q$ of order $q-1$ defined by
\[
\begin{array}{crccl}
\chi_\frak p:&\f_q=\frak o_{\Bbb Q(\zeta_{q-1})}/\frak p&\longrightarrow&T \vspace{1mm}\cr
&t+\frak p&\longmapsto&t,&\in T.
\end{array}
\]
For each $a\in \Bbb Z$, the Gauss sum of $\chi_{\frak p}^a$ is
\[
g(\chi_{\frak p}^a)=\sum_{t\in\langle\zeta_{q-1}\rangle}\chi_{\frak p}^a(t)\zeta_p^{\text{Tr}_{q/p}(t+\frak p)}\in \frak o_{\Bbb Q(\zeta_{p(q-1)})}.
\]
Let $\wp$ be the unique prime of $\frak o_{\Bbb Q(\zeta_{p(q-1)})}$ lying above $\frak p$. $\wp$ is totally ramified over $\frak p$ with ramification index $e(\wp\mid \frak p)=p-1$.

For an integer $a\ge 0$ with base $p$ expansion $a=a_0+a_1p+\cdots$, $0\le a_i\le p-1$, define $s(a)=a_0+a_1+\cdots$ and $\gamma(a)=a_0!a_1!\cdots$. The Stickelberger congruence states that for $1\le a\le q-2$,
\begin{equation}\label{2.1}
\frac{g(\chi_{\frak p}^{-a})}{(\zeta_p-1)^{s(a)}}\equiv\frac{-1}{\gamma(a)}\pmod\wp.
\end{equation}

\subsection{$p$-adic expansion of $|Z(f)|$}\

For $\uu=(u_1,\dots,u_m)\in\Bbb N^n$, let $|\uu|=u_1+\cdots+u_n$. If $\x=(x_1,\dots,x_n)$ is an $n$-tuple of elements from a commutative ring, we define $\x^{\uu}=x_1^{u_1}\cdots x_n^{u_n}$. Let $U_d=\{\uu\in\Bbb N^n:|\uu|\le d\}$ and consider
\[
f=\sum_{\uu\in U_d}a_{\uu}\X^{\uu}\in\f_q[X_1,\dots,X_n],
\]
where $\X=(X_1,\dots,X_n)$. We write $\sum_{\boldsymbol u}$ and $\prod_{\boldsymbol u}$ for $\sum_{\boldsymbol u\in U_d}$ and $\prod_{\boldsymbol u\in U_d}$, respectively. By \cite[($5'$)]{Ax-AJM-1964}, we have
\begin{equation}\label{2.2}
q|Z(f)|=\sum_{i:U_d\to\{0,\cdots,q-1\}}\Bigl(\prod_{\uu}\alpha_{\uu}^{i(\uu)}\Bigr)\Bigl(\prod_{\uu}c_{i(\uu)}\Bigr)\sum_{\T\in T^{n+1}}\T^{\sum_{\uu}i(\uu)(1,\uu)},
\end{equation}
where $\alpha_{\uu}\in T$ is such that 
\begin{equation}\label{2.2.1}
a_{\uu}=\alpha_{\uu}+\frak p,
\end{equation}
and 
\begin{equation}\label{2.3}
c_i=
\begin{cases}
1&\text{if}\ i=0,\vspace{2mm} \cr
\displaystyle -\frac q{q-1}&\text{if}\ i=q-1,\vspace{2mm} \cr
\displaystyle \frac 1{q-1}g(\chi_\frak p^{-i})&\text{if}\ 0<i<q-1.
\end{cases}
\end{equation}
By \eqref{2.1}, we have $\nu_\wp(c_i)=s(i)$ for all $0\le i\le q-1$. From the proof in \cite[\S3]{Ax-AJM-1964}, we know that
\begin{equation}\label{2.4}
\nu_\wp\biggl(\Bigl(\prod_{\uu}c_{i(\uu)}\Bigr)\sum_{\T\in T^{n+1}}\T^{\,\sum_{\uu}i(\uu)(1,\uu)}\biggr)\ge m(p-1)\Bigl\lceil\frac nd\Bigr\rceil
\end{equation}
for all $i:U_d\to \{0,\dots,q-1\}$, where $\nu_\wp$ is the $\wp$-adic valuation. In fact, \eqref{2.4} implies \eqref{1.1} immediately. In what follows, we will reprove \eqref{2.4}, and we will focus on those $i$ for which the equal sign holds in \eqref{2.4}. 

When $\sum_{\uu}i(\uu)(1,\uu)\not\equiv(0,\dots,0)\pmod{q-1}$, 
\[
\sum_{\T\in T^{n+1}}\T^{\,\sum_{\uu}i(\uu)(1,\uu)}=0.
\]
When $\sum_{\uu}i(\uu)(1,\uu)=(0,\dots,0)$,
\[
\text{LSH of \eqref{2.4}}\ge\nu_\wp(q^{n+1})=m(p-1)(n+1)>m(p-1)\Bigl\lceil\frac nd\Bigr\rceil.
\]
Therefore, we assume that $\sum_{\uu}i(\uu)(1,\uu)\equiv(0,\dots,0)\pmod{q-1}$ but $i\ne0$ ($i(\uu)\ne 0$ for at least one $\uu\in U_d$). Let $k$ be the number of nonzero components of $\sum_{\uu}i(\uu)\uu$. Then 
\begin{equation}\label{2.4.1}
\sum_{\T\in T^{n+1}}\T^{\,\sum_{\uu}i(\uu)(1,\uu)}=(q-1)^{k+1}q^{n-k},
\end{equation}
and
\begin{align*}
\text{LSH of \eqref{2.4}}\,&=\nu_\wp\biggl(\Bigl(\prod_{\uu}c_{i(\uu)}\Bigr)(q-1)^{k+1}q^{n-k}\biggr)\\
&=\sum_{\uu}s(i(\uu))+m(p-1)(n-k)\\
&\ge m(p-1)\Bigl\lceil\frac kd\Bigr\rceil+m(p-1)(n-k)\numberthis \label{2.5}\\
&= m(p-1)\Bigl(\Bigl\lceil\frac kd\Bigr\rceil+n-k\Bigr)\\
&\ge m(p-1)\Bigl\lceil\frac nd\Bigr\rceil.\numberthis \label{2.6}
\end{align*}
In the above, inequality \eqref{2.6} is straightforward; inequality \eqref{2.5} was proved in \cite{Ax-AJM-1964} and will be explained below. First, we have

\begin{fac}\label{F2.1}
When $d\ge 2$, the equal sign in \eqref{2.6} holds if and only if (i) $k=n$, or (ii) $k=n-1$ and $d\mid n-1$.
\end{fac}

Next, we determine the necessary and sufficient conditions for the equal sign to hold in \eqref{2.5}. We have
\[
d\sum_{\uu}i(\uu)\ge\sum_{\uu}i(\uu)|\uu|\ge k(q-1).
\]
Since $\sum_{\uu}i(\uu)\equiv 0\pmod{q-1}$, we have
\begin{equation}\label{2.7}
\sum_{\uu}i(\uu)\ge(q-1)\Bigl\lceil\frac kd\Bigr\rceil.
\end{equation}
For $a\in\{0,1,\dots,q-1\}$ with base $p$ expansion $a=a_0+a_1p+\cdots+a_{m-1}p^{m-1}$, $0\le a_j\le p-1$, define
\[
\tau(a)=a_{m-1}+a_0p+\cdots+a_{m-2}p^{m-1}.
\]
Then \eqref{2.7} remains true with $i(\uu)$ replaced by $\tau(i(\uu))$. Therefore,
\begin{equation}\label{2.8}
m(q-1)\Bigl\lceil\frac kd\Bigr\rceil\le\sum_{h=0}^{m-1}\sum_{\uu}\tau^h(i(\uu))=\frac{q-1}{p-1}\sum_{\uu}s(i(\uu)),
\end{equation}
i.e.,
\[
\sum_{\uu}s(i(\uu))\ge m(p-1)\Bigl\lceil\frac kd\Bigr\rceil,\tag{$2.5'$}
\]
which is the same as \eqref{2.5}.

\begin{fac}\label{F2.2}
The equal sign in {\rm ($2.5'$)} holds if and only if 
\begin{equation}\label{2.9}
\sum_{\uu}i(\uu)^{(j)}=(p-1)\Bigl\lceil\frac kd\Bigr\rceil\quad \text{for all}\ 0\le j\le m-1,
\end{equation}
where $(i(\uu)^{(0)},\dots,i(\uu)^{(m-1)})$ are the base $p$ digits of $i(\uu)$.
\end{fac}

\begin{proof}
First note that the equal sign in ($2.5'$) holds if and only if 
\begin{equation}\label{2.10}
\sum_{\uu}\tau^h(i(\uu))=(q-1)\Bigl\lceil\frac kd\Bigr\rceil\quad\text{for all}\ 0\le h\le m-1.
\end{equation}
We prove that \eqref{2.9} is equivalent to \eqref{2.10}.

($\Rightarrow$) Assume that \eqref{2.9} holds. Then for each $0\le h\le m-1$ we have
\[
\begin{split}
\sum_{\uu}\tau^h(i(\uu))\,&=\sum_{\uu}\tau^h\Bigl(\sum_{j=0}^{m-1}i(\uu)^{(j)}p^j\Bigr)=\sum_{\uu}\sum_{j=0}^{m-1}i(\uu)^{(j)}\tau^h(p^j)\cr
&=\sum_{j=0}^{m-1}\Bigl(\sum_{\uu}i(\uu)^{(j)}\Bigr)\tau^h(p^j)=(p-1)\Bigl\lceil\frac kd\Bigr\rceil\sum_{j=0}^{m-1}\tau^h(p^j)\cr
&=(p-1)\Bigl\lceil\frac kd\Bigr\rceil(1+p+\cdots+p^{m-1})=(q-1)\Bigl\lceil\frac kd\Bigr\rceil.
\end{split}
\]

($\Leftarrow$) Assume that \eqref{2.10} holds. Since
\[
\begin{split}
\tau^h(i(\uu))\,&=\tau\bigl(\tau^{h-1}(i(\uu))\bigr)=p\tau^{h-1}(i(\uu))-\bigl(\tau^{h-1}(i(\uu))\bigr)^{(m-1)}(q-1)\cr
&=p\tau^{h-1}(i(\uu))-i(\uu)^{(m-h)}(q-1),
\end{split}
\]
where $m-h$ is taken modulo $m$, we have
\[
\begin{split}
(q-1)\Bigl\lceil\frac kd\Bigr\rceil\,&=\sum_{\uu}\tau^h(i(\uu))=\sum_{\uu}\Bigl(p\tau^{h-1}(i(\uu))-i(\uu)^{(m-h)}(q-1)\Bigr)\cr
&=p(q-1)\Bigl\lceil\frac kd\Bigr\rceil-(q-1)\sum_{\uu}i(\uu)^{(m-h)},
\end{split}
\]
i.e.,
\[
\sum_{\uu}i(\uu)^{(m-h)}=(p-1)\Bigl\lceil\frac kd\Bigr\rceil.
\]
\end{proof}

We assume that $d\ge 2$ (to avoid trivial situations). 

\begin{defn}\label{D2.1}
Let $\mathcal I$ be the set of functions $i:U_d\to\{0,\dots,q-1\}$ such that
\begin{itemize}
  \item [(i)] each component of $\sum_{\uu}i(\uu)\uu$ is a positive multiple of $q-1$;
  \item [(ii)] $\sum_{\uu}i(\uu)^{(j)}=(p-1)\lceil n/d\rceil$ for all $0\le j\le m-1$.
\end{itemize} 
If $d\mid n-1$, let $\mathcal I'$ be the set of functions $i:U_d\to\{0,\dots,q-1\}$ such that
\begin{itemize}
  \item [(i)] one of the component of $\sum_{\uu}i(\uu)\uu$ is $0$ and the other components are all positive multiples of $q-1$;
  \item [(ii)] $\sum_{\uu}i(\uu)^{(j)}=(p-1)(n-1)/d$ for all $0\le j\le m-1$.
\end{itemize}
If $d\nmid n-1$, define $\mathcal I'=\emptyset$.
\end{defn}

By Facts~\ref{F2.1} and \ref{F2.2}, the equal sign in \eqref{2.4} holds if and only if $i\in\mathcal I\cup\mathcal I'$. Therefore by \eqref{2.2} and \eqref{2.4.1},
\begin{equation}\label{2.11}
\begin{split}
q|Z(f)|\equiv\,&\sum_{i\in\mathcal I\cup\mathcal I'}\Bigl(\prod_{\uu}\alpha_{\uu}^{i(\uu)}\Bigr)\Bigl(\prod_{\uu}c_{i(\uu)}\Bigr)\sum_{\T\in T^{n+1}}\T^{\,\sum_{\uu}i(\uu)(1,\uu)}\pmod{q^{\lceil n/d\rceil}\wp}\cr
=\,&\sum_{i\in\mathcal I}\Bigl(\prod_{\uu}\alpha_{\uu}^{i(\uu)}\Bigr)\Bigl(\prod_{\uu}c_{i(\uu)}\Bigr)(q-1)^{n+1}+
\sum_{i\in\mathcal I}\Bigl(\prod_{\uu}\alpha_{\uu}^{i(\uu)}\Bigr)\Bigl(\prod_{\uu}c_{i(\uu)}\Bigr)(q-1)^nq.
\end{split}
\end{equation}
We know that
\begin{equation}\label{2.12}
c_{i(\uu)}\equiv\frac{(\zeta_p-1)^{s(i(\uu))}}{\gamma(i(\uu))}\pmod{(\zeta_p-1)^{s(i(\uu))}\wp}.
\end{equation}
(\eqref{2.12} is obvious when $i(\uu)=0$, and follows from \eqref{2.3} and \eqref{2.1} when $1<i(\uu)<q-1$. When $i(\uu)=q-1$, \eqref{2.12} is easily verified directly.) Also note that 
\begin{equation}\label{2.13}
\begin{split}
p=\prod_{j=1}^{p-1}(\zeta_p^j-1)\,&=(\zeta_p-1)^{p-1}\prod_{j=1}^{p-1}\frac{\zeta_p^j-1}{\zeta_p-1}\cr
&\equiv(\zeta_p-1)^{p-1}(p-1)!\pmod{\zeta_p-1)^p}\cr
&\equiv-(\zeta_p-1)^{p-1}\pmod{\zeta_p-1)^p}.
\end{split}
\end{equation}
Now combining \eqref{2.11} -- \eqref{2.13} gives
\begin{equation}\label{2.14}
\begin{split}
q|Z(f)|\equiv\,&\sum_{i\in\mathcal I}\Bigl(\prod_{\uu}\alpha_{\uu}^{i(\uu)}\Bigr)\frac{(\zeta_p-1)^{m(p-1)\lceil n/d\rceil}}{\prod_{\uu}\gamma(i(\uu))}(q-1)^{n+1}\cr
&+\sum_{i\in\mathcal I'}\Bigl(\prod_{\uu}\alpha_{\uu}^{i(\uu)}\Bigr)\frac{(\zeta_p-1)^{m(p-1)(\lceil n/d\rceil-1)}}{\prod_{\uu}\gamma(i(\uu))}(q-1)^{n}q\pmod{q^{\lceil n/d\rceil}\wp}\cr
\equiv\,&q^{\lceil n/d\rceil}(-1)^{n+m\lceil n/d\rceil}\Bigl[-\sum_{i\in\mathcal I}\prod_{\uu}\frac{\alpha_{\uu}^{i(\uu)}}{\gamma(i(\uu))}+(-1)^m\sum_{i\in\mathcal I'}\prod_{\uu}\frac{\alpha_{\uu}^{i(\uu)}}{\gamma(i(\uu))}\Bigr]\cr
&\kern 7cm \pmod{q^{\lceil n/d\rceil}\wp}.
\end{split}
\end{equation}
Let 
\begin{equation}\label{2.14.1}
\mathcal E(f)=(-1)^{n+m\lceil n/d\rceil}\Bigl[-\sum_{i\in\mathcal I}\prod_{\uu}\frac{\alpha_{\uu}^{i(\uu)}}{\gamma(i(\uu))}+(-1)^m\sum_{i\in\mathcal I'}\prod_{\uu}\frac{\alpha_{\uu}^{i(\uu)}}{\gamma(i(\uu))}\Bigr],
\end{equation}
and write \eqref{2.14} as
\begin{equation}\label{2.16.1}
|Z(f)|\equiv q^{\lceil n/d\rceil-1}\mathcal E(f)\pmod{q^{\lceil n/d\rceil-1}\wp}.
\end{equation}
Since $\mathcal E(f)\in\Bbb Q(\zeta_{q-1})$, \eqref{2.16.1} gives
\begin{equation}\label{2.14.2}
|Z(f)|\equiv q^{\lceil n/d\rceil-1}\mathcal E(f)\pmod{q^{\lceil n/d\rceil-1}p}.
\end{equation}
Since $|Z(f)|\in\Bbb Z$, there exists $N\in\Bbb Z$ such that
\begin{equation}\label{2.14.3}
\mathcal E(f)\equiv N\pmod p.
\end{equation}
Taking images of both sides of \eqref{2.14.3} in $\{x\in\Bbb Q(\zeta_{q-1}):\nu_\frak p(x)\ge 0\}/\frak p=\f_q$, we have 
\begin{equation}\label{2.14.4}
E(f)=N\quad \text{(in $\f_q$)},
\end{equation}
where
\begin{equation}\label{2.14.5}
E(f)=(-1)^{n+m\lceil n/d\rceil}\Bigl[-\sum_{i\in\mathcal I}\prod_{\uu}\frac{a_{\uu}^{i(\uu)}}{\gamma(i(\uu))}+(-1)^m\sum_{i\in\mathcal I'}\prod_{\uu}\frac{a_{\uu}^{i(\uu)}}{\gamma(i(\uu))}\Bigr].
\end{equation}
In fact, $E(f)\in\f_p$ because of \eqref{2.14.4}.

To summarize, we have the following theorem.

\begin{thm}\label{T2.2}
Let $n\ge 1$, $d\ge 2$, and 
\[
f=\sum_{\uu\in U_d}a_{\uu}\X^{\uu}\in\f_q[X_1,\dots,X_n],
\]
where $\X=(X_1,\dots,X_n)$. We have
\begin{equation}\label{2.15}
|Z(f)|\equiv q^{\lceil n/d\rceil-1}E(f)\pmod {q^{\lceil n/d\rceil-1}p},
\end{equation}
where $E(f)$ is given in \eqref{2.14.5}. In particular, $\nu_p(|Z(f)|)\ge m(\lceil n/d\rceil-1)+1$ if and only if $E(f)=0$.
\end{thm}

\begin{rmk}\label{R2.3}\rm
$E(f)$ is a homogeneous polynomial of degree $(q-1)\lceil n/d\rceil$ over $\f_p$ in the coefficients of $f$. In general, this expression is not explicit because $\mathcal I$ and $\mathcal I'$ are not. In the next section, we explore several special cases where $E(f)$ can be made explicit.
\end{rmk}


\section{Applications to Reed-Muller Codes}

\subsection{Reed-Muller codes}\

For a prime power $q=p^m$ and integers $n, d$ with $n>0$ and $0\le d\le n(q-1)$, the $q$-ary Reed-Muller code $R_q(d,n)$ is defined as
\begin{equation}\label{3.7}
R_q(d,n)=\bigl\{f\in\f_q[X_1,\dots,X_n]: \deg f\le d,\ \deg_{X_j}f\le q-1,\ 1\le j\le n\bigr\}.
\end{equation}
(For convenience, we define $R_q(-1,n)=\{0\}$.) It is known that \cite[Result~1]{Ding-Key}
\begin{equation}\label{dim}
\dim_{\f_q}R_q(d,n)=\sum_{j\le\lfloor d/q\rfloor}(-1)^j\binom nj\binom{d-qj+n}n.
\end{equation}
For each $f\in R_q(d,n)$, its (Hemming) weight is $|f|=q^n-|Z(f)|$.  
The weight enumerator of $R_q(d,n)$ is not known except for the following special cases.
\begin{itemize}
  \item [(i)] $d\le 2$ or $d\ge n(q-1)-3$. (For $d=2$ and $q=2$, see \cite[Ch.\,15, \S2]{MacWilliams-Sloane}; for $d=2$ and $q$ general, use the well known classification of quadratic forms over $\f_q$. For $d\ge n(q-1)-3$, note that the dual of $R_q(d,n)$ is $R_q(d',n)$, where $d'=n(q-1)-1-d\le 2$.)
  \item [(ii)] $q=2$ and $n\le 8$ (\cite{Kasami-Tokura-Azumi-1976, Sugino-Ienage-Tokura-Kasami-1971}).
  \item [(iii)] $q=2$, $n=9$, $d=3$ (\cite{Sugita-Kasami-Fujiwara-IEEE-IT-1996}).
\end{itemize} 

\smallskip
For $t\ge 0$, let 
\[
N_q(d,n;t)=\bigl|\{f\in R_q(d,n):\nu_p(|f|)\ge t\}\bigr|.
\]
Ax's theorem implies that $N_q(d,n;t)=|R_q(d,n)|$ for $t\le m(\lceil n/d\rceil-1)$. We will use Theorem~\ref{T2.2} to determine $N_q(d,n;t)$ with $t=m(\lceil n/d\rceil-1)+1$ in several cases; such formulas provide new information concerning the weight enumerators of the Reed-Muller codes involved. The cases we consider share a common assumption that $(p-1)\lceil n/d\rceil=2$, that is, $p=2$ and $\lceil n/d\rceil=2$, or $p=3$ and $\lceil n/d\rceil=1$. Under this assumption, for each $i\in\mathcal I$ (Definition~\ref{D2.1}),
\begin{equation}\label{3.13}
\sum_{\uu}i(\uu)^{(j)}=2\quad\text{for all $0\le j\le m-1$}.
\end{equation}

\subsection{The case $q=2^m$ and $d=n/2$}\label{S3.2}\

Assume that $q=2^m$, $n\ge 4$ is even, and $d=n/2$. Let $f=\sum_{\uu\in U_{n/2}}a_{\uu}\X^{\uu}\in\f_q[X_1,\dots,X_n]$. Since $d\nmid n-1$, $\mathcal I'=\emptyset$ in Definition~\ref{D2.1}. Hence 
\begin{equation}\label{3.0}
E(f)=(-1)^{n+1}\sum_{i\in\mathcal I}\prod_{\uu}\frac{a_{\uu}^{i(\uu)}}{\gamma(i(\uu))}.
\end{equation}

If $i\in\mathcal I$, then
\[
\sum_{\uu\in U_d}i(\uu)=\sum_{\uu}\sum_{j=0}^{m-1}i(\uu)^{(j)}2^j=2\sum_{j=0}^{m-1}2^j=2(q-1).
\]
Since
\[
n(q-1)\le\sum_{\uu\in U_{n/2}}i(\uu)|\uu|\le\frac n2\sum_{\uu\in U_{n/2}}i(\uu)=n(q-1),
\]
we have $|\uu|=n/2$ for all $\uu\in U_{n/2}$ with $i(\uu)>0$ and we have
\begin{equation}\label{3.1}
\sum_{|\uu|=n/2}i(\uu)\uu=(q-1,\dots,q-1).
\end{equation}

\begin{lem}\label{L3.1}
$i\in\mathcal I$ if and only if there exist $\uu_j,\vv_j\in\{0,1\}^n$, $0\le j\le m-1$, with $|\uu_j|=|\vv_j|=n/2$, $\uu_j+\vv_j=(1,\dots,1)$, such that for all $0\le j\le m-1$,
\begin{equation}\label{3.2}
\begin{cases}
i(\uu_j)^{(j)}=i(\vv_j)^{(j)}=1,\vspace{1mm}\cr
i(\uu)^{(j)}=0&\text{if}\ \uu\in U_{n/2}\setminus\{\uu_j,\vv_j\}.
\end{cases}
\end{equation}
\end{lem}

\begin{proof}
($\Rightarrow$) By Definition~\ref{D2.1}, 
\begin{equation}\label{3.3}
\sum_{|\uu|=n/2}i(\uu)^{(j)}=2\quad\text{for all}\ 0\le j\le m-1.
\end{equation}
Choose $\uu_{m-1}, \vv_{m-1}\in U_{n/2}$ with $|\uu_{m-1}|=|\vv_{m-1}|=n/2$ such that $i(\uu_{m-1})^{(m-1)}=i(\vv_{m-1})^{(m-1)}=1$. Since
\[
\begin{split}
(2^m-1)(1,\dots,1)\,&=\sum_{|\uu|=n/2}i(\uu)\uu\ge i(\uu_{m-1})\uu_{m-1}+i(\vv_{m-1})\vv_{m-1}\cr
&\ge 2^{m-1}(\uu_{m-1}+\vv_{m-1}),
\end{split}
\]
it follows that $\uu_{m-1}+\vv_{m-1}\le(1,\dots,1)$, that is, $\uu_{m-1},\vv_{m-1}\in\{0,1\}^n$ and $\uu_{m-1}+\vv_{m-1}=(1,\dots,1)$. For any $\uu\in U_{n/2}$ with $|\uu|=n/2$ and $\uu\ne \uu_{m-1},\vv_{m-1}$, we have $i(\uu)^{(m-1)}=0$ by \eqref{3.3}. 

Now we have
\[
\sum_{|\uu|=n/2}\sum_{j=0}^{m-2}i(\uu)^{(j)}2^j\uu=(2^{m}-1)(1,\dots,1)-2^{m-1}(1,\dots,1)=(2^{m-1}-1)(1,\dots,1).
\]
By the same argument, there exist $\uu_{m-2}, \vv_{m-2}\in\{0,1\}^n$ with $|\uu_{m-2}|=|\vv_{m-2}|=n/2$ and $\uu_{m-2}+\vv_{m-2}=(1,\dots,1)$ such that $i(\uu_{m-2})^{(m-2)}=i(\vv_{m-2})^{(m-2)}=1$ and $i(\uu)^{(m-2)}=0$ for all $\uu$ with $|\uu|=n/2$ and $\uu\ne\uu_{m-2},\vv_{m-2}$. Continuing this way, we have $\uu_j,\vv_j$, $0\le j\le m-1$, with the desired property.

\medskip
($\Leftarrow$) For each $0\le j\le m-1$,
\[
\sum_{\uu} i(\uu)^{(j)}=i(\uu_j)^{(j)}+i(\vv_j)^{(j)}=2=(p-1)\lceil n/d\rceil.
\]
Also,
\[
\begin{split}
\sum_{\uu}i(\uu)\uu\,&=\sum_{\uu}\Bigl(\sum_{j=0}^{m-1}i(\uu)^{(j)}2^j\Bigr)\uu=\sum_{j=0}^{m-1}2^j(\uu_j+\vv_j)\cr
&=\Bigl(\sum_{j=0}^{m-1}2^j\Bigr)(1,\dots,1)=(q-1)(1,\dots,1).
\end{split}
\]
Hence $i\in\mathcal I$.
\end{proof}

It follows from Lemma~\ref{L3.1} that
\begin{equation}\label{3.4}
\begin{split}
\sum_{i\in\mathcal I}\prod_{\uu}a_{\uu}^{i(\uu)}\,&=\sum_{\substack{\{\uu_0,\vv_0\},\dots,\{\uu_{m-1},\vv_{m-1}\}\cr \uu_j,\vv_j\in\{0,1\}^n,\,|\uu_j|=|\vv_j|=n/2\cr \uu_j+\vv_j=(1,\dots,1)}}a_{\uu_0}a_{\vv_0}(a_{\uu_1}a_{\vv_1})^2\cdots(a_{\uu_{m-1}}a_{\vv_{m-1}})^{2^{m-1}}\cr
&=\biggl(\sum_{\substack{\{\uu,\vv\}\cr \uu,\vv\in\{0,1\}^n,\, |\uu|=|\vv|=n/2\cr \uu+\vv=(1,\dots,1)}}a_{\uu}a_{\vv}\biggr)^{1+2+\cdots+2^{m-1}}.
\end{split}
\end{equation}
Combining Theorem~\ref{T2.2}, \eqref{3.0} and \eqref{3.4} gives the following corollary.

\begin{cor}\label{C3.2}
Let $q=2^m$ and $n\ge 4$ be even. Let 
\[
f=\sum_{\uu\in U_{n/2}}a_{\uu}\X^{\uu}\in\f_q[X_1,\dots,X_n].
\]
Then $v_2(|Z(f)|)\ge m+1$ if and only if
\begin{equation}\label{3.5}
\sum_{\substack{\{\uu,\vv\}\cr \uu,\vv\in\{0,1\}^n,\, |\uu|=|\vv|=n/2\cr \uu+\vv=(1,\dots,1)}}a_{\uu}a_{\vv}=0.
\end{equation}
\end{cor} 

Replacing each $a_{\uu}$ in \eqref{3.5} by an indeterminate $Y_{\uu}$, we obtain a quadratic form
\[
Q=\sum_{\substack{\{\uu,\vv\}\cr \uu,\vv\in\{0,1\}^n,\, |\uu|=|\vv|=n/2\cr \uu+\vv=(1,\dots,1)}}Y_{\uu}Y_{\vv}
\]
in $N=\binom n{n/2}$ indeterminates over $\f_q$. Order the indeterminates in a row $\Y=(Y_{\uu}:\uu\in\{0,1\}^n,\ |\uu|=n/2)$ such that the indices $\uu$ and $\uu^c:=(1,\dots,1)-\uu$ appear in positions symmetric to the center of the row. Then 
\[
Q=\Y A\Y^t,
\]
where
\[
A=
\left[
\begin{matrix}
0&\Delta_{N/2}\cr
0&0
\end{matrix}\right]_{N\times N}
\] 
and $\Delta_{N/2}$ is defined in \eqref{Delta}.
By \cite[Theorem~6.32]{LN}, the number of roots of $Q$ in $\f_q^N$ is 
\begin{equation}\label{3.6}
q^{N-1}+(q-1)q^{\frac 12N-1}.
\end{equation}

\begin{cor}\label{C3.3}
Let $q=2^m$ and $n\ge 4$ be even. Then
\begin{equation}\label{3.8}
N_q(n/2,n;\,m+1)=\Bigl(q^{\binom n{n/2}-1}+(q-1)q^{\frac 12\binom n{n/2}-1}\Bigr)q^{\dim_{\f_q}R_q(n/2,n)-\binom n{n/2}},
\end{equation}
where
\begin{equation}\label{3.9}
\dim_{\f_q}R_q(n/2,n)=\sum_{j\le\lfloor n/2q\rfloor}(-1)^j\binom nj\binom{\frac{3n}2-qj}n.
\end{equation}
\end{cor}

\begin{proof}\eqref{3.8} follows from Corollary~\ref{C3.2} and \eqref{3.6}; \eqref{3.9} follows from \eqref{dim}.
\end{proof}

In the remaining three subsections, arguments and computations are similar to those in Subsection~\ref{S3.2}. Therefore, a fair amount of details is omitted. 


\subsection{The case $q=2$ and $n/2\le d\le n-2$}\

Assume that $q=2$, $n\ge 4$, and $n/2\le d\le n-2$. Let $f=\sum_{\uu\in U_d}a_{\uu}\X^{\uu}\in\f_2[X_1,\dots,X_n]$. Then $\mathcal I'=\emptyset$ and 
\[
E(f)=(-1)^{n+1}\sum_{i\in\mathcal I}\prod_{\uu}\frac{a_{\uu}^{i(\uu)}}{\gamma(i(\uu))}.
\]
Moreover, $i\in\mathcal I$ if and only if there exist $\uu_0,\vv_0\in U_d\cap\{0,1\}^n$ with $\uu+\vv=(1,\dots,1)$ such that 
\[
\begin{cases}
i(\uu_0)=i(\vv_0)=1,\vspace{1mm}\cr
i(\uu)=0&\text{for all}\ \uu\in U_d\setminus\{\uu_0,\vv_0\}. 
\end{cases}
\]
Consequently, 
\[
\sum_{i\in\mathcal I}\prod_{\uu}\frac{a_{\uu}^{i(\uu)}}{\gamma(i(\uu))}=\sum_{\substack{\{\uu_0,\vv_0\}\cr \uu_0,\vv_0\in\{0,1\}^n,\,|\uu_0|,|\vv_0|\in[n-d,d]\cr \uu_0+\vv_0\ge(1,\dots,1)}}a_{\uu_0}a_{\vv_0}.
\]
Thus $\nu_2(|Z(f)|)\ge 2$ if and only if $(a_{\uu}:\uu\in\{0,1\}^n,\ |\uu|\in[n-d,d])$ is a root of the quadratic form
\[
Q=\sum_{\substack{\{\uu,\vv\}\cr \uu,\vv\in\{0,1\}^n,\,|\uu|,|\vv|\in[n-d,d]\cr \uu+\vv\ge(1,\dots,1)}}Y_{\uu}Y_{\vv}.
\]
Order the indeterminates of $Q$ in a row $\Y=(Y_{\uu}:\uu\in\{0,1\}^n,\ |\uu|\in[n-d,d])$ such that $|\uu|$ is increasing and the indices $\uu$ and $\uu^c:=(1,\dots,1)-\uu$ appear in positions symmetric to the center of the row. Then 
\[
Q=\Y A\Y^T,
\]
where
\[
A=\left[\kern2.5cm\begin{matrix}
&& & & 1\cr
&& &\cdot&*\cr
&&\cdot&\cdot &\cdot\cr
&\cdot&\cdot&&\cdot\cr
1&*&\cdot&\cdot&*\cr
*&*&\cdot&\cdot&*\cr
&\cdot&\cdot&&\cdot\cr
&&\cdot&\cdot &\cdot\cr
&& &\cdot&*\cr
&& & & *
\end{matrix}\right]_{N\times N},
\qquad N=\sum_{j=n-d}^d\binom nj.
\]
(The unmarked entries of $A$ are all $0$.) There exists $P\in\text{GL}(N,\f_2)$ such that 
\[
PAP^T=\left[\begin{matrix}0&\Delta_{N/2}\cr 0&0\end{matrix}\right].
\]
Therefore the number of roots of $Q$ in $\f_2^N$ is $2^{N-1}+2^{\frac 12N-1}$ \cite[Theorem~6.32]{LN}. 

\begin{cor}\label{C3.9}
For $n\ge 4$ and $n/2\le d\le n-2$,
\[
N_2(d,n;2)=2^{\binom n0+\cdots+\binom nd-1}+2^{2^{n-1}-1}.
\]
\end{cor}


\subsection{The case $q=3^m$ and $d=n$}\

Assume that $q=3^m$, $n\ge 2$, and $d=n$. Let $f=\sum_{\uu\in U_n}a_{\uu}\X^{\uu}\in\f_q[X_1,\dots,X_n]$. Then $\mathcal I'=\emptyset$. Moreover, $i\in\mathcal I$ if and only if there exist $\uu_j,\vv_j\in\{0,1,2\}^n$, $0\le j\le m-1$, with $|\uu_j|=|\vv_j|=n$ and $\uu_j+\vv_j=(2,\dots,2)$ such that for all $0\le j\le m-1$, 
\[
\begin{cases}
i(\uu_j)^{(j)}=i(\vv_j)^{(j)}=1&\text{if}\ \uu_j\ne\vv_j,\cr
i(\uu_j)^{(j)}=2&\text{if}\ \uu_j=\vv_j,\cr
i(\uu)^{(j)}=0&\text{if}\ \uu\in U_n\setminus\{\uu_j,\vv_j\}.
\end{cases}
\] 
We have 
\[
E(f)=(-1)^{n+m+1}\biggl(\sum_{\substack{\{\uu,\vv\}\cr \uu,\vv\in\{0,1,2\},\,|\uu|=|\vv|=n\cr \uu+\vv=(2,\dots,2)}}a_{\uu}a_{\vv}\biggr)^{1+3+\cdots+3^{m-1}}.
\]
Thus $\nu_3(|Z(f)|)\ge 1$ if and only if $(a_{\uu}:\uu\in\{0,1,2\}^n,\ |\uu|=n)$ is a root of the quadratic form
\[
Q=\sum_{\substack{\{\uu,\vv\}\cr \uu,\vv\in\{0,1,2\},\,|\uu|=|\vv|=n\cr \uu+\vv=(2,\dots,2)}}Y_{\uu}Y_{\vv}.
\]
Order the indeterminates of $Q$ in a row $\Y=(Y_{\uu}:\uu\in\{0,1,2\}^n,\ |\uu|=n)$ such that the indices $\uu$ and $\uu^c:=(2,\dots,2)-\uu$ appear in positions symmetric to the center of the row. Then
\[
Q=\Y A\Y^T,
\]
where 
\[
A=\left[\begin{matrix} 0&\Delta_{(N+1)/2}\cr 0&0\end{matrix}\right]_{N\times N},\qquad N=\sum_{j\le n/2}\binom nj\binom{n-j}{n-2j}.
\]
The number of roots of $Q$ in $\f_q^N$ is $q^{N-1}$ \cite[Theorem~6.27]{LN}. 

\begin{cor}\label{C3.7} 
Let $q=3^m$ and $n\ge 2$. Then
\[
N_q(n,n;1)=q^{\dim_{\f_q}R_q(n,n)-1},
\]
where
\[
\dim_{\f_q}R_q(n,n)=\sum_{j\le\lfloor n/q\rfloor}(-1)^j\binom nj\binom{2n-qj}n.
\]
\end{cor}


\subsection{The case $q=3$ and $n\le d\le 2n$}\

Assume that $q=3$, $n\ge 2$, and $n\le d\le 2n$. Let $f=\sum_{\uu\in U_d}a_{\uu}\X^{\uu}\in\f_q[X_1,\dots,X_n]$. Then $\mathcal I'=\emptyset$. Moreover, $i\in\mathcal I$ if and only if there exist $\uu_0,\vv_0\in\{0,1,2\}^n$ with $\uu_0\equiv\vv_0\pmod 2$ and $\uu_0+\vv_0\ge(2,\dots,2)$ such that 
\[
\begin{cases}
i(\uu_0)=i(\vv_0)=1&\text{if}\ \uu_0\ne \vv_0,\cr
i(\uu_0)=2&\text{if}\ \uu_0=\vv_0,\cr
i(\uu)=0&\text{if}\ \uu\in U_d\setminus\{\uu_0,\vv_0\}.
\end{cases}
\]
We have
\[
E(f)=(-1)^{n+m+1}\sum_{\substack{\{\uu,\vv\}\cr
\uu,\vv\in\{0,1,2\}^n,\,|\uu|,|\vv|\in[2n-d,d]\cr
\uu\equiv\vv\,\text{(mod\,$2$)},\,\uu+\vv\ge (2,\dots,2)}} a_{\uu}a_{\vv}.
\]
Thus $\nu_3(|Z(f)|)\ge 1$ if and only if $(a_{\uu}:\uu\in\{0,1,2\}^n,\,|\uu|\in[2n-d,d])$ is a root of the quadratic form
\[
Q=\sum_{\substack{\{\uu,\vv\}\cr
\uu,\vv\in\{0,1,2\}^n,\,|\uu|,|\vv|\in[2n-d,d]\cr
\uu\equiv\vv\,\text{(mod\,$2$)},\,\uu+\vv\ge (2,\dots,2)}} Y_{\uu}Y_{\vv}.
\]
Order the indeterminates of $Q$ in a row $\Y=(Y_{\uu}:\uu\in\{0,1,2\}^n,\,|\uu|\in[2n-d,d])$ such that $|\uu|$ is increasing and the indices $\uu$ and $\uu^c:=(2,\dots,2)-\uu$ appear in positions symmetric to the center of the row. Then
\[
Q=\Y A\Y^T,
\]
where 
\[
A=\left[\kern2.5cm\begin{matrix}
&& & & 1\cr
&& &\cdot&*\cr
&&\cdot&\cdot &\cdot\cr
&\cdot&\cdot&&\cdot\cr
1&*&\cdot&\cdot&*\cr
&*&\cdot&\cdot&*\cr
&&\cdot&\cdot &\cdot\cr
&& &\cdot&\cdot\cr
&& & & *
\end{matrix}\right]_{N\times N},
\] 
\[
N=\sum_{j=2n-d}^d\bigl|\bigl\{\uu\in\{0,1,2\}^n: |\uu|=j\bigr\}\bigr|.
\]
There exists $P\in\text{GL}(N,\f_3)$ such that 
\[
PAP^T=\left[\begin{matrix} 0&\Delta_{(N+1)/2}\cr 0&0\end{matrix}\right].
\] 
Hence the number of roots of $Q$ in $\f_3^N$ is $3^{N-1}$ \cite[Theorem~6.27]{LN}. 

\begin{cor}\label{C3.11}
Let $n\ge 2$ and $n\le d\le 2n$. Then
\[ 
N_3(d,n;1)=3^{\dim_{\f_3}R_3(d,n)-1},
\]
where
\[
\dim_{\f_3}R_3(d,n)=\sum_{j\le\lfloor d/3\rfloor}(-1)^j\binom nj\binom{d-3j+n}n.
\]
\end{cor}


\end{document}